\documentclass[11pt]{amsart}
\usepackage{geometry}               
\usepackage{graphicx}
\usepackage{tikz}
\usepackage{amssymb}
\usepackage{pdfsync}
\usepackage{hyperref}
\usepackage{verbatim} 
\usepackage{epstopdf}
\usepackage{tikz}
\tikzstyle{every picture}=[>=latex]
\usepackage{esint} 

\def\R{\mathbb{R}}

\def\Co{\mathbb{C}}

\newcommand{\mM}{{\mathcal M}} 
\renewcommand{\d}{\,{\rm d}} 
\newcommand{\graffe}[1]{\left\{#1\right\}} 
\newcommand{\insieme}[2]{\graffe{#1 \colon #2}} 
\newcommand{\pvector}[1]{
  \begin{pmatrix}
    #1
  \end{pmatrix}} %
\newcommand{\ddirac}[1]{
  \,\boldsymbol{\delta}\!\pvector{#1}\!} %
\newcommand{\epsi}{\varepsilon} 
\newcommand{\tonde}[1]{\left(#1\right)} 
\newcommand{\ip}[1]{\langle #1 \rangle} 
\newcommand{\abs}[1]{\left|#1\right|}
\newcommand{\intersection}{\cap} 

\newtheorem{theorem}{Theorem}
\newtheorem{step}{Step}

\newtheorem{proposition}[theorem]{Proposition}

\numberwithin{equation}{section}

\title[Sharp Fourier restriction theory]{Some recent progress on sharp Fourier restriction theory}

\author{Damiano Foschi}
\address{
        Damiano Foschi\\
        Dipartimento di Matematica e Informatica\\
        Universit\`{a} di Ferrara\\
        Via Machiavelli 30, 
        44121 Ferrara, Italy}
\email{damiano.foschi@unife.it}

\author{Diogo Oliveira e Silva}
\address{
        Diogo Oliveira e Silva\\
        Hausdorff Center for Mathematics\\
        53115 Bonn, Germany}
\email{dosilva@math.uni-bonn.de}
\thanks{\it The second author was partially supported by the Hausdorff Center for Mathematics.}

\date{\today}                                           
\begin{document}

\begin{abstract}
The purpose of this note is to discuss several results that have been obtained in the last decade in the context of sharp adjoint Fourier restriction/Strichartz inequalities. 
Rather than aiming at full generality, we focus on several concrete examples of underlying manifolds with large groups of symmetries,
which sometimes allows for simple geometric proofs.
We mention several open problems along the way, and include an appendix on integration on manifolds using delta calculus.
\end{abstract}


\maketitle

\section{Introduction}

 Curvature causes the Fourier transform to decay. 
This observation links geometry to analysis, and lies at the base of several topics of modern harmonic analysis.
Given the long history of the subject, it is perhaps surprising that the possibility of restricting the Fourier transform to curved submanifolds of Euclidean space was not observed until the late sixties. 
 The Fourier transform of an integrable function is uniformly continuous, and as such, can be restricted to any subset. 
On the other hand, the Fourier transform of a square-integrable function is again square-integrable, and in view of Plancherel's theorem no better properties can be expected. 
In particular, restricting the Fourier transform of a square-integrable function to a set of zero Lebesgue measure is meaningless. 
The question is what happens for intermediate values of $1<p<2$.

 It is not hard to check that the Fourier transform of a radial function in $L^p(\R^d)$ defines a continuous function away from the origin whenever $1\leq p<\frac{2d}{d+1}$, see for instance~\cite{StSh4}. 
The corresponding problem for non-radial functions is considerably more delicate. 
 To introduce it, let $\mM$ be a smooth compact hypersurface in $\R^{d}$, endowed with a surface-carried measure $\d\mu=\psi \d\sigma$.
Here $\sigma$ denotes the surface measure of $\mM$, and the function $\psi$ is smooth and non-negative. Given $1<p<2$, for which exponents $q$ does the {\it a priori} inequality
\begin{equation}\label{restriction}
\Big(\int_\mM |\widehat{f}(\xi)|^q \d\mu_\xi\Big)^{\frac1q}\leq C \|f\|_{L^p(\R^{d})}
\end{equation}
hold? 
A complete answer for $q=2$ is given by the celebrated Tomas--Stein inequality. 
\begin{theorem}[\cite{St, To}]\label{TomasStein}
Suppose $\mM$ has non-zero Gaussian curvature at each point of the support of $\mu$. 
Then the restriction inequality \eqref{restriction} holds for $q=2$ and $1\leq p\leq\frac{2d+2}{d+3}$. 
\end{theorem}
\noindent 
The range of exponents is sharp, since no $L^p\to L^2(\mu)$ restriction can hold for $\mM$ if~$p>\frac{2d+2}{d+3}$. 
This is shown via the famous Knapp example, which basically consists of testing  the inequality dual to \eqref{restriction} against the characteristic function of a small cap on $\mM$. 
Moreover, some degree of curvature is essential, as there can be no meaningful restriction to a hyperplane except in the trivial case when $p=1$ and $q=\infty$. 
An example to keep in mind is that of the unit sphere $\mathbb{S}^{d-1}=\{\xi\in\R^{d}: |\xi|=1\}$, with constant positive Gaussian curvature.
However, nonvanishing Gaussian curvature is a strong assumption that can be replaced by the nonvanishing of some principal curvatures, at the expense of decreasing the range of admissible exponents $p$.

 The question of what happens for values of $q<2$ is the starting point for the famous restriction conjecture. 
One is led by dimensional analysis and Knapp-type examples
to guess that the correct range for estimate \eqref{restriction} to hold is $1\leq p< \frac{2d}{d+1}$ and $q\leq (\frac{d-1}{d+1})p'$, where $p'=\frac p{p-1}$ denotes the dual exponent. 
This is depicted in Figure \ref{fig:restriction}.
Note that the endpoints of this relation are the trivial case $(p, q)=(1,\infty)$, and $p,q\to \frac{2d}{d+1}$. 

\begin{figure}[htbp]
  \centering

  \begin{tikzpicture} [scale = 5]
    \fill[black!5] (0, 0) rectangle (1, 1);
    \shade[left color = yellow, right color = green!50] (0.625, 1) -- (0.625, 0.625) -- (0.7, 0.5) -- (0.7, 1) -- cycle;
    \fill[green] (0.7, 1) -- (0.7, 0.5) -- (1, 0) -- (1, 1) -- cycle;
    \draw[->] (0, 0) -- (1.05, 0) node[right] {$\frac1p$};
    \draw[->] (0, 0) -- (0, 1.05) node[above] {$\frac1q$};
    \draw (0, 1) node[left] {$1$};
    \draw[dotted] (1, 1) -- (0, 0) node[below left] {$0$}; 
    \draw[dotted] (1, 1) -- (1, 0) node[below] {$1$};
    \draw[dotted] (0.625, 1) -- (0.625, 0);
    \draw[<-, black!60, text = black] (0.625, 0) .. controls (0.625, -0.06) .. (0.6, -0.06) node[left = -3] {$\frac{d+1}{2d}$};
    \draw[dotted] (0.7, 1) -- (0.7, 0);
    \draw[<-, black!60, text = black] (0.7, 0) .. controls (0.7, -0.06) .. (0.725, -0.06) node[right = -3] {$\frac{d+3}{2d+2}$};
    \draw[dotted] (1, 0.5) -- (0, 0.5) node[left] {$\frac12$};
    \draw[dotted] (0.625, 0.625) -- (0, 0.625) node[left] {$\frac{d+1}{2d}$};
    \draw[very thick, blue] (1, 0.5) -- (0.7, 0.5);
    \draw[<-, black!60, text = black] (0.7, 0.5) .. controls (0.625, 0.3) .. (0.5, 0.3) node[fill = white, left] {Tomas--Stein};
    \fill[red] (0.7, 0.5) circle [radius = 0.01];
    \draw[<-, black!60, text = black] (0.625, 0.625) .. controls (0.55, 0.8) .. (0.5, 0.8) node[fill = white, left, align=center] {restriction\\conjecture};
    \draw[red] (0.625, 0.625) circle [radius = 0.01];
  \end{tikzpicture}
  
  \caption{Range of exponents for the restriction problem}
  \label{fig:restriction}
\end{figure}
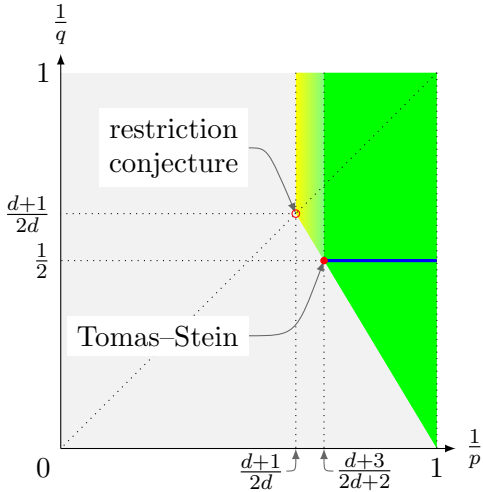

\noindent Despite tremendous effort and very promising partial progress, the restriction conjecture is only known to hold for $d=2$. 
The restriction conjecture implies the Kakeya conjecture and is implied by the Bochner--Riesz conjecture. 
Multilinear versions of the restriction and Kakeya conjectures have been established by Bennett, Carbery and Tao \cite{BCT}, and played a crucial role in the very recent work of Bourgain and Demeter \cite{BD} on $\ell^2$ decoupling.
For more on the restriction problem, and its relation to other prominent problems in modern harmonic analysis, we recommend the works \cite{St, Wo} and \mbox{especially \cite{T}.}

Tomas--Stein type restriction estimates are very much related to Strichartz estimates for linear partial differential equations of dispersion type. 
Let us illustrate this point in two cases,
that of solutions $u=u(x,t)$ to the homogeneous Schr\"odinger equation 
\begin{equation}\label{Schroedinger}
i\partial_t u=\Delta u,
\end{equation}
and that of solutions to the homogeneous wave equation 
\begin{equation}\label{Wave}
\partial^2_{t} u=\Delta u.
\end{equation}
In both situations, $(x,t)\in\R^{d+1}$.
The following theorem was originally proved by Strichartz.

\begin{theorem}[\cite{Str}]\label{Strichartz}
Let $d\geq 1$. Then there exists a constant $S>0$ such that 
\begin{equation}\label{Strichartzineq}
\|u\|_{L^{2+\frac 4d}(\R^{d+1})}\leq S \|f\|_{L^2(\R^d)},
\end{equation}
whenever $u$ is the solution of \eqref{Schroedinger} with initial data $u(x,0)=f(x)$.
If $d\geq 2$, then there exists a constant $W>0$ such that
\begin{equation}\label{StrichartzineqW}
\|u\|_{L^{2+\frac{4}{d-1}}(\R^{d+1})}\leq W \|(f,g)\|_{\dot{H}^{\frac12}(\R^d)\times \dot{H}^{-\frac12}(\R^d)},
\end{equation}
whenever $u$ is the solution of \eqref{Wave} with initial data $u(x,0)=f(x)$ and $\partial_t u(x,0)=g(x)$.
\end{theorem}
\noindent A hint that Theorems \ref{TomasStein} and \ref{Strichartz} might be related comes from the numerology of the exponents: The Strichartz exponent $2+\frac4d$ coincides with the dual of the Tomas--Stein exponent in dimension $d+1$.
It turns out that
Strichartz estimates for the  Schr\"odinger equation correspond to restriction estimates on the paraboloid, 
whereas Strichartz estimates for the wave equation correspond to restriction estimates on the cone.
Note that the Gaussian curvature of the cone is identically zero
because one of its principal curvatures vanishes. 
This in turn translates into estimate \eqref{StrichartzineqW} holding for the Strichartz exponent in one lower dimension.
Perhaps more significantly, neither of these manifolds is compact. 
However, they exhibit some scale invariance properties that enable a reduction to the compact setting.
We shall return to this important point later in our discussion.

 In this note, we are interested in extremizers and optimal constants for sharp variants of restriction and Strichartz-type inequalities.
 Apart from their intrinsic mathematical interest and elegance, such sharp inequalities often allow for various refinements of existing inequalities.
The following are natural questions, which in particular can be posed for inequalities \eqref{restriction}, \eqref{Strichartzineq} and \eqref{StrichartzineqW}:
\begin{enumerate}
\item What is the value of the optimal constant?
\item Do extremizers exist?
\begin{enumerate}
\item If so, are they unique, possibly after applying the symmetries of the problem?
\item If not, what is the mechanism responsible for this lack of compactness?
\end{enumerate}
\item How do extremizing sequences behave?
\item What are some qualitative properties of extremizers?
\item What are necessary and sufficient conditions for a function to be an extremizer?
\end{enumerate}
Questions of this flavor have been asked in a variety of situations, and in the context of classical inequalities from Euclidean harmonic analysis  go back at least to the seminal work of Beckner \cite{Be} for the Hausdorff--Young inequality, and Lieb \cite{Li} for the Hardy--Littlewood--Sobolev inequality.
In comparison, sharp Fourier restriction inequalities have a relatively short history, with the first works on the subject going back to Kunze \cite{K}, Foschi \cite{F2} and Hundertmark--Zharnitsky \cite{HZ}.

Works addressing the existence of extremizers for inequalities of Fourier restriction type tend to be a {\it tour de force} in classical analysis,
using a variety of sophisticated techniques:
bilinear estimates  and refined estimates in $X^{s,b}$-type spaces \cite{CS, OS, Q, Ra},
concentration compactness arguments tailored to the specific problem in question \cite{FLS, JPS, JSS, K, Sh, Sh2, Sh3},
variants and generalizations of the Br\'ezis--Lieb lemma from functional analysis \cite{FVV, FVV2, FLS},
Fourier integral operators \cite{CS, FLS},
symmetrization techniques \cite{CS},
variational, perturbative and spectral analysis \cite{CS,  DMR, FLS, OS},
regularity theory for equations with critical scaling and additive combinatorics \cite{CS2},
among others.
In contrast, a full characterization of extremizers has been given in a few selected cases using much more elementary methods.
This is due to the presence of a large underlying group of symmetries which allows for several simplifications that ultimately reduce the problem to a simple geometric observation. 
We will try to illustrate this point in the upcoming sections.

Before doing so, let us briefly comment on some approaches that have been developed in the last decades in order to establish Tomas--Stein type Fourier restriction inequalities. 
For the sake of brevity, we specialize our discussion to inequalities \eqref{restriction} and \eqref{Strichartzineq}, but a more general setting should be kept in mind.
If $T$ denotes the restriction operator $f\mapsto \widehat{f}|_\mM$, then its adjoint $T^*$, usually called the extension operator, is given by 
$g\mapsto \widehat{g\mu}$, where the Fourier transform of the measure $g\mu$ is given by
\begin{equation}\label{gsigmahat}
\widehat{g\mu}(x)=\int_\mM g(\omega) e^{-i x\cdot \omega} \d\mu_\omega.\;\;\;(x\in\R^d)
\end{equation}
 The Tomas--Stein inequality at the endpoint $(p,q)=(\frac{2d+2}{d+3},2)$ is equivalent to the extension estimate
\begin{equation}\label{extensionineq}
\|\widehat{g\mu}\|_{L^{\frac{2d+2}{d-1}}(\R^d)}\leq C \|g\|_{L^2(\mM,\mu)}.
\end{equation}
If $g\in L^2(\mM,\mu)$, then the composition $T^*T(g)$ is well-defined, and a computation shows that 
$$T^*T: g\mapsto g\ast\widehat{\mu}.$$
Since the operator norms satisfy $\|T\|^2=\|T^*\|^2=\|T^*T\|$, the study of these three operators is equivalent, even if the goal is to obtain sharp inequalities and determine optimal constants.
So we focus on the operator $T^*T$. 
Boundedness of $T^*T(g)$ is only ensured if 
the Fourier transform $\widehat{\mu}(x)$
exhibits some sort of decay as $|x|\to\infty$. 
This in turn is a consequence of the principle of stationary phase, see \cite{St, Wo},
since the nonvanishing curvature of $\mM$ translates into a nondegenerate Hessian for the phase function of the oscillatory integral given by \eqref{gsigmahat} with $g\equiv 1$. 
This is the starting point for the original argument of Tomas \cite{To}, which was then extended to the endpoint $\frac{2d+2}{d-1}$ by embedding $T^*T$ into an analytic family of operators and invoking Stein's complex interpolation theorem.
It is hard not to notice the parallel between the operator $T^*T$ and the averaging operator~$f\mapsto f\ast\mu,$
whose $L^p$ improving properties can be established via the same proof on the Fourier side, see for instance \cite[pp. 370-371]{St}.

A second method to prove restriction estimates goes back to the work of Ginibre and Velo \cite{GV}. 
It consists of introducing a time parameter and treating the extension operator 
as an evolution operator.
Two key ingredients for this approach are the Hausdorff--Young inequality and fractional integration in the form of the Hardy--Littlewood--Sobolev inequality.\footnote{Interestingly, the full restriction conjecture on $\R^2$ can be proved with a combination of Hausdorff--Young and Hardy--Littlewood--Sobolev, see for instance \cite[pp. 412-414]{St}.}
These methods are more amenable to the needs of the partial differential equations community, 
and for instance allow to treat the case of mixed norm spaces.

In the special cases when the dual exponent $p'$ is an even integer, one can devise yet another proof which comes from the world of bilinear estimates, see for instance \cite{Bourg1, FK, KM1}. One simply rewrites the left-hand side of inequality \eqref{extensionineq} as an $L^2$ norm, and appeals to Plancherel in order to reduce the problem to a multilinear convolution estimate. 
For instance, if $d=3$, then $p'=\frac{2\cdot 3+2}{3-1}=4$, and 
$$\|\widehat{g\mu}\|^2_{L^4(\R^3)}
=\|\widehat{g\mu}\;\widehat{g\mu}\|_{L^2(\R^3)}
=\|g\mu\ast g\mu\|_{L^2(\R^3)}.$$
Similarly, if $d=2$, then $p'=\frac{2\cdot 2+2}{2-1}=6$, and 
$$\|\widehat{g\mu}\|^3_{L^6(\R^2)}=\|g\mu\ast g\mu\ast g\mu\|_{L^2(\R^2)}.$$
In the bilinear case, the pointwise inequality
$$|g\mu\ast g\mu|\leq |g|\mu\ast |g|\mu$$
then reveals that one can restrict attention to nonnegative functions. This observation can greatly simplify matters, as it reduces an oscillatory problem to a question of geometric integration over a specific manifold.
Furthermore, an application of the Cauchy--Schwarz inequality with respect to an appropriate measure implies the pointwise inequality
$$|g\mu\ast g\mu|^2\leq (|g|^2\mu\ast |g|^2\mu)  (\mu\ast\mu).$$
A good understanding of the convolution measure $\mu\ast\mu$ 
becomes a priority.
Given integrable functions $g, h\in L^1(\mM,\mu)$, the convolution $g\mu\ast h\mu$ is a finite measure defined on the Borel subsets $E\subset\R^d$ by
$$(g\mu\ast h\mu)(E)=\int_{\mM\times \mM} \chi_E(\eta+\zeta) g(\eta) h(\zeta) \d\mu_\eta \d\mu_\zeta.$$ 
It is clear that this measure is supported on the Minkowski sum $\mM+\mM$.
In most situations of interest when some degree of curvature is present, one can check that $g\mu\ast h\mu$ is absolutely continuous with respect to Lebesgue measure on $\R^{d}$.
In such cases, the measure $g\mu\ast h\mu$ can be identified with its Radon--Nikodym derivative with respect to Lebesgue measure, and for almost every $\xi\in\R^d$ we have that
\begin{equation}\label{deltacalc}
(g\mu\ast h\mu)(\xi)=\int_{\mM\times \mM}\ddirac{\xi-\eta-\zeta} g(\eta) h(\zeta) \d\mu_\eta \d\mu_\zeta.
\end{equation} 
Here $\ddirac{\cdot}$ denotes the $d$-dimensional Dirac delta distribution. 
As we shall exemplify in the course of the paper, expression \eqref{deltacalc} turns out to be very useful for computational purposes.\\

\noindent {\bf Overview.}
As we hope to have made apparent already, this note is meant as a short survey of a restricted part of the topic of sharp Strichartz and Fourier extension inequalities.
In \S \ref{sec:noncompact} we deal with noncompact surfaces, and discuss the cases of the paraboloid and the cone, where a full characterization of extremizers is known, and the cases of the hyperboloid and a quartic perturbation of the paraboloid, where extremizers fail to exist.
Most of the material in this section is contained in the works \cite{C, F2, OSQ, Q, Q2}. 
In \S \ref{sec:compact}  we discuss the case of spheres. A full characterization of extremizers at the endpoint is only known in the case of $\mathbb{S}^2$. We mention some recent partial progress in the case of the circle~$\mathbb{S}^1$, and observe how the methods in principle allow to refine some related inequalities. 
In particular, we obtain a new sharp extension inequality on $\mathbb{S}^1$ in the mixed norm space~$L_{\text{rad}}^6 L_{\text{ang}}^2(\R^2)$.
Most of the material in \S  \ref{sec:compact} is contained in the works \cite{COS, CFOST, Co, F, OST}. 
We leave some final remarks to \S  \ref{sec:unify}, where we hint at a possible unifying picture for the results that have been discussed.
Finally, we include an Appendix with a brief introduction to integration on manifolds using delta calculus.\\

\noindent {\bf Remarks and further references.}
The style of this note is admittedly informal. 
In particular, some objects will not be rigorously defined, and most results will not be precisely formulated.
None of the material is new, with the exception of the results in \S \ref{sec:MixedNorm} and a few observations that we have not been able to find in the literature.
The subject is becoming more popular, as shown by the increasing number of works that appeared in the last five years. 
We have attempted to give a rather complete set of references, which includes several interesting works
 \cite{BBI, BBJP, BJ, BR, BS, B, HS, J, OS2, OR, OR2, Se} that will not be discussed here.
Given its young age, there are plenty of open problems in the area.
Our contribution is to provide some more.

\section{Noncompact surfaces}\label{sec:noncompact}

\subsection{Paraboloids}\label{sec:paraboloids}
Given $d\geq 1$, let us consider the $d$-dimensional paraboloid
$$\mathbb{P}^d=\{(\xi,\tau)\in\R^{d+1}: \tau={|\xi|^2}\},$$
equipped with projection measure
$$\mu_d(\xi,\tau)=\ddirac{\tau-{|\xi|^2}}\d\xi \d\tau.$$
The validity of an $L^2(\mu_d)\to L^{2+\frac4d}$ extension estimate follows from Strichartz inequality~\eqref{Strichartzineq} for the Schr\"odinger equation as discussed before. 
Extremizers for this inequality are known to exist in all dimensions \cite{Sh}, and to be Gaussians in low dimensions \cite{BBCH, F2, HZ}.
Extremizers are conjectured to be Gaussians in all dimensions \cite{HZ}, see also \cite{CQ}.
Let us specialize to the case $d=2$ and follow mostly \cite{F2}.
In view of identity \eqref{deltacalc}, which itself is a consequence of formula \eqref{eq:dddef} from the Appendix, the convolution of projection measure on the two-dimensional paraboloid $\mathbb{P}^2$ is given by
$$(\mu_2\ast\mu_2) (\xi,\tau)
=\int_{\R^2\times\R^2} 
\ddirac{\xi-\eta-\zeta \\ \tau-|\eta|^2-|\zeta|^2} 
\d \eta \d \zeta
=\int_{\R^2}\ddirac{\tau-|\eta|^2-|\xi-\eta|^2}\d \eta.$$
Changing variables and computing in  polar coordinates according to \eqref{changevars1}, we have that
$$
(\mu_2\ast\mu_2) (\xi,\tau)
=\int_{\R^2}\ddirac{\tau-\frac{|\xi|^2}2-2|\eta|^2}\d \eta
=\frac{\pi}2\chi\Big(\tau\geq \frac{|\xi|^2}2\Big).$$
We arrive at the crucial observation that the convolution measure $\mu_2\ast \mu_2$ defines a function which is not only uniformly bounded, but also constant in the interior of its support $\{2\tau\geq |\xi|^2\}$. See \cite[Lemma 3.2]{F2} for an alternative proof that uses the invariance of $\mu_2\ast \mu_2$ under Galilean transformations and parabolic dilations. 
A successive application of Cauchy--Schwarz and H\"older's inequality finishes the argument. Indeed, the pointwise bound
\begin{equation}\label{pointwise}
|(f\mu_2\ast f\mu_2)(\xi,\tau)|^2
\leq 
(\mu_2\ast\mu_2)(\xi,\tau)
(|f|^2\mu_2\ast |f|^2\mu_2)(\xi,\tau)
\end{equation}
follows from an application of the Cauchy--Schwarz inequality with respect to the measure 
$$\ddirac{\xi-\eta-\zeta  \\ \tau-|\eta|^2-|\zeta|^2} \d \eta \d \zeta.$$
Integrating inequality \eqref{pointwise} over $\R^{2+1}$, an application of H\"older's inequality then reveals 
\begin{equation}\label{HolderafterCS}
\|f\mu_2\ast f\mu_2\|_{L^2(\R^3)}^2\leq \|\mu_2\ast\mu_2\|_{L^\infty(\R^3)} \|f\|_{L^2(\R^2)}^4.
\end{equation}
It is possible to turn both inequalities simultaneously into an equality. The conditions for equality in \eqref{pointwise} translate into a functional equation
\begin{equation}\label{functionaleqGauss}
f(\eta)f(\zeta)=F(\eta+\zeta,|\eta|^2+|\zeta|^2)
\end{equation}
which should hold for some complex-valued function $F$ defined on the support of the convolution $\mu_2\ast\mu_2$, and almost every point $(\eta,\zeta)$.  
An example of a solution to \eqref{functionaleqGauss} is given by the Gaussian function $f(\eta)=\exp(-|\eta|^2)$ and the corresponding $F(\xi,\tau)=\exp(-\tau)$.
All other solutions are obtained from this one by applying a symmetry of the Schr\"odinger equation, see \cite[Proposition 7.15]{F2}.
That they turn inequality \eqref{HolderafterCS} into an equality follows from the fact that the convolution $\mu_2\ast\mu_2$ is constant inside its support.

The one-dimensional case $L^2(\mu_1)\to L^6$  admits a similar treatment. 
The threefold convolution $\mu_1\ast\mu_1\ast\mu_1$ is given by a constant function in the interior of its support~$\{3\tau\geq \xi^2\}$, and the corresponding functional equation can be solved by similar methods. 
Gaussians are again seen to be the unique extremizers.

Alternative approaches are available:
Hundertmark--Zharnitski \cite{HZ} based their analysis on a novel representation of the Strichartz integral  $\|\widehat{f\sigma}\|_{L^4}^4$, and Bennett et al. \cite{BBCH} identified a monotonicity property of such integrals under a certain quadratic heat-flow.

\subsection{Cones}
Given $d\geq 2$, consider the (one-sheeted) cone
$$\Gamma^d=\{(\xi,\tau)\in\R^{d+1}: \tau=|\xi|>0\},$$
equipped with its Lorentz invariant measure
$$\nu_d(\xi,\tau)
=\ddirac{\tau-|\xi|} \frac{\d\xi \d\tau}{|\xi|}
=2\ddirac{\tau^2-|\xi|^2}\;{\chi(\tau>0)}{\d\xi \d\tau}.$$
The second identity is a consequence of formula \eqref{scalarmult} from the Appendix.
The validity of an $L^2(\nu_d)\to L^{2+\frac4{d-1}}$ extension estimate follows from Strichartz inequality \eqref{StrichartzineqW} for the wave equation as discussed before. 
Extremizers for the cone are known 
to exist in all dimensions \cite{Ra}, and 
to be exponentials in low dimensions \cite{C, F2}.
Let us specialize to the case $d=3$.
The convolution of the Lorentz invariant measure on the three-dimensional cone $\Gamma^3$ is given by
$$(\nu_3\ast\nu_3)(\xi,\tau)
=\int_{\R^3}\frac{\ddirac{\tau-|\eta|-|\xi-\eta|}}{|\eta||\xi-\eta|} \d \eta.$$
Given $\xi\in\R^3$, we write a generic vector $\eta\in\R^3$ in spherical coordinates, so that
$\d \eta=\rho^2\sin\theta \d\rho \d\theta \d\varphi,$
where $\rho=|\eta|\geq 0$, $\theta\in[0,\pi]$ is the angle between $\xi$ and $\eta$, and $\varphi\in[0,2\pi]$ is an angular variable. 
Setting $\sigma=|\xi-\eta|$,
the Jacobian of the change of variables $\eta\rightsquigarrow(\rho,\sigma,\varphi)$ into bipolar coordinates, see Figure \ref{fig:bipolar} below, is given by
$$\d \eta=\frac{\rho\sigma}{|\xi|} \d\rho \d\sigma \d\varphi.$$

\begin{figure}[htbp]
  \centering
  
  \begin{tikzpicture} [scale = 1]
    \clip (-2.01, -2.51) rectangle (5.01, 2.51);
 
    \foreach \r in {0.5, 1, ..., 5} {
      \draw[thin, blue!40] (0, 0) circle [radius = \r];
      \draw[thin, green!60] (3, 0) circle [radius = \r];
    }

    \draw[thin, ->] (-2, 0) -- (5, 0);
    \draw[thin, ->] (0, -2.5) -- (0, 2.5);

    \draw[very thick, red] (1.5, 0) circle [x radius = 2.5, y radius = 2];

    \draw[thick, ->] (0,0) -- (2.333, 1.886) node[midway, above] {$\rho$};
    \draw[thick, ->] (3,0) -- (2.333, 1.886) node[midway, right] {$\sigma$};

    \fill (0, 0) circle [radius = 0.05] node[below left] {$0$};
    \fill (3, 0) circle [radius = 0.05] node[below left] {$\xi$};
    \fill (2.333, 1.886) circle [radius = 0.05] node[above] {$\eta$};

  \end{tikzpicture}

  \caption{Bipolar coordinates}
  \label{fig:bipolar}
\end{figure}
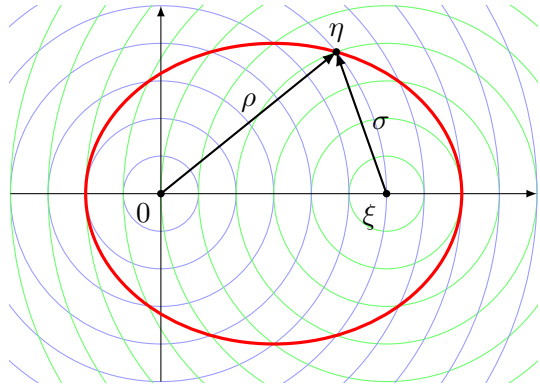

\noindent Letting $a=\rho-\sigma$ and $b=\rho+\sigma$, we invoke the change of variables formula \eqref{changevars1} to further compute
\begin{multline*}
(\nu_3\ast\nu_3)(\xi,\tau)
=2\pi\int_{\substack{|\rho-\sigma|\leq |\xi| \\ \rho+\sigma\geq |\xi|}} \frac{\ddirac{\tau-\rho-\sigma}}{\rho\sigma} \frac{\rho\sigma}{|\xi|} \d\rho \d\sigma=\\
=\frac{\pi}{|\xi|} \int_{\substack{|a|\leq |\xi| \\ b\geq |\xi|}} \ddirac{\tau-b} \d a \d b=2\pi\chi(\tau\geq |\xi|).
\end{multline*}
This again defines a constant function inside its support, and a combination of Cauchy--Schwarz and H\"older as before establishes the sharp inequality. The characterization of extremizers follows from the analysis of the functional equation
\begin{equation}\label{functionaleqExp}
f(\eta)f(\zeta)=F(\eta+\zeta,|\eta|+|\zeta|),
\end{equation}
 which yields as a particular solution the exponential function $f(\eta)=\exp(-|\eta|)$ and the corresponding $F(\xi,\tau)=\exp(-\tau)$.
All other solutions are obtained from this one by applying a symmetry of the wave equation, see \cite[Proposition 7.23]{F2}, and 
the lower dimension case~$L^2(\nu_2)\to L^6$ admits a similar treatment. 

\subsection{Hyperboloids}
We now switch to the (one-sheeted) hyperboloid
$$\mathbb{H}^d=\{(\xi,\tau)\in\R^{d+1}: \tau=\sqrt{1+|\xi|^2}\},$$
equipped with the Lorentz invariant measure
$$\lambda_d(\xi,\tau)=\ddirac{\tau-\sqrt{1+|\xi|^2}}\frac{ \d\xi \d\tau}{\sqrt{1+|\xi|^2}}
=2\ddirac{\tau^2-|\xi|^2-1}\;{\chi(\tau>0)}{\d\xi \d\tau}.$$
As established in \cite{Str}, an $L^2(\lambda_d)\to L^{p'}$ extension estimate holds provided
\begin{align}
2+\frac 4d \leq p'\leq 2+\frac4{d-1},\text{ if }d>1,\label{HypNum}\\
6\leq p'<\infty, \text{ if }d=1.\notag
\end{align}
Note that the lower and upper bounds in the exponent range \eqref{HypNum} correspond to the cases of the paraboloid and the cone, respectively.
We focus on the case $d=2$ and $p'=4$, and take advantage of Lorentz symmetries to compute the convolution $\lambda_2\ast\lambda_2$. Along the vertical axis of the hyperboloid $\mathbb{H}^2$,
$$(\lambda_2\ast\lambda_2)(0,\tau)
=\int_{\R^2}\ddirac{\tau-2\sqrt{1+|\eta|^2}} \frac{\d \eta}{1+|\eta|^2}
=\frac{2\pi}{\tau}\chi(\tau\geq 2).$$
Lorentz invariance forces the convolution $\lambda_2\ast\lambda_2$ to be constant along the level sets of the function $\tau^2-|\xi|^2$. 
As a consequence,
\begin{equation}\label{hypconv}
(\lambda_2\ast\lambda_2)(\xi,\tau)=\frac{2\pi}{\sqrt{\tau^2-|\xi|^2}}\chi(\tau\geq \sqrt{4+|\xi|^2}).
\end{equation}
Contrary to the previous cases, this no longer defines a constant function inside its support. 
Since it is uniformly bounded (by $\pi$), the argument can still be salvaged to yield a sharp extension inequality. Extremizers for this inequality, however, do not exist. We shall observe a similar phenomenon in the case of perturbed paraboloids, considered in the next subsection, and postpone a more detailed discussion until then.
\subsection{Perturbations}
Let us start with a brief discussion of  a specific instance of a comparison principle from \cite{OSQ} that proved useful in establishing sharp inequalities for perturbed paraboloids.
Let $(\mathbb{P}^2,\mu_2)$ denote the two-dimensional paraboloid considered in \S \ref{sec:paraboloids}, and let $\widetilde{\mu}_2$ denote the projection measures on the surface
$$\{(\xi,\tau)\in\R^{2+1}: \tau=|\xi|^2+|\xi|^4\}.$$
Then the pointwise inequality 
\begin{equation}\label{comparison}
(\widetilde{\mu}_2\ast\widetilde{\mu}_2)\Big(\xi,\frac{|\xi|^2}2+\frac{|\xi|^4}8+\tau\Big)
\leq (\mu_2\ast\mu_2)\Big(\xi,\frac{|\xi|^2}2+\tau\Big)
\end{equation}
holds for every $\tau>0$ and $\xi\in\R^d$,
and is strict at almost every point of the support of the measure $\widetilde{\mu}_2\ast\widetilde{\mu}_2$. 
The feature of the function $\xi\mapsto|\xi|^4$ that makes this possible is {\it convexity}. Any nonnegative, continuously differentiable, strictly convex function would do, see \cite[Theorem 1.3]{OSQ} for a precise version of this comparison principle which holds in all dimensions.  
A sharp $L^2(\widetilde{\mu}_2)\to L^4$ extension inequality can be obtained by concatenating Cauchy--Schwarz and H\"older as before, and extremizers do not exist because of the strict inequality in \eqref{comparison}.
Heuristically, extremizing sequences are forced to place their mass in arbitrarily small neighborhoods of the region where the convolution $\widetilde{\mu}_2\ast\widetilde{\mu}_2$ attains its global maximum, for otherwise they would not come close to attaining the sharp constant.
The analysis of the cases of equality in \eqref{comparison} reveals that this region has zero Lebesgue measure, and this forces any extremizing sequence to concentrate.

\section{Compact surfaces}\label{sec:compact}
\subsection{Spheres}
Consider the endpoint  Tomas--Stein extension inequality on the sphere 
\begin{equation}\label{sphererestriction}
\Big(\int_{\R^{d}} |\widehat{f\sigma}(x)|^{p'} \d x\Big)^{\frac1{p'}}\leq C \|f\|_{L^2(\mathbb{S}^{d-1},\sigma)},
\end{equation}
where  $\sigma=\sigma_{d-1}$ denotes surface measure on $\mathbb{S}^{d-1}$ and $p'=\frac{2d+2}{d-1}$.
Extremizers for inequality \eqref{sphererestriction} were first shown to exist when $d=3$ in \cite{CS}. 
The precise form of nonnegative extremizers was later determined in \cite{F}, and they turn out to be constant functions.
See also \cite{FLS} for a conditional existence result in higher dimensions.
It seems natural to conjecture that extremizers should be constants in all dimensions.
Spheres are antipodally symmetric compact manifolds, and this brings in some additional difficulties which can already be observed at the level of convolution measures. Indeed, formula
\eqref{scalarmult} from the Appendix implies
$$\sigma_{d-1}(\xi)=\ddirac{1-|\xi|}\d\xi=2\ddirac{1-|\xi|^2}\d\xi.$$
Invoking \eqref{prodrule}, and then \eqref{scalarmult} once again, we have
$$(\sigma_{d-1}\ast\sigma_{d-1})(\xi)
=2\int_{\mathbb{S}^{d-1}}\ddirac{1-|\xi-\omega|^2} \d\sigma_\omega
=\frac{2}{|\xi|}\int_{\mathbb{S}^{d-1}}\ddirac{2\frac{\xi}{|\xi|}\cdot\omega-|\xi|}\d\sigma_\omega.
$$
Computing in polar coordinates according to \eqref{changevars1}, one concludes
\begin{equation}\label{2foldsphconv}
(\sigma_{d-1}\ast\sigma_{d-1})(\xi)
=\frac{V_{d-2}}{|\xi|}\Big(1-\frac{|\xi|^2}4\Big)_+^{\frac{d-3}2},
\end{equation}
where $V_{d-2}$ denotes the surface measure of $\mathbb{S}^{d-2}$. 
When $d=3$, 
and consequently $p'=4$, the inequality in question is equivalent to a 4-linear estimate.
Moreover, 
the last factor in \eqref{2foldsphconv} simplifies and one is left with
\begin{equation}\label{sphconv}
(\sigma_2\ast\sigma_2)(\xi)=\frac{2\pi}{|\xi|}\chi(|\xi|\leq2).
\end{equation}
Expression \eqref{sphconv} blows up at the origin $\xi=0$, which prevents a straightforward adaptation of the methods from the previous section. In particular, the interaction between antipodal points causes difficulties that prevented the local analysis from \cite{CS} to identify the global extremizers of the problem. The work \cite{F} resolves this issue in a simple manner, using the following geometric feature of the sphere: If the sum of three unit vectors $\omega_1,\omega_2,\omega_3\in\mathbb{S}^d$ is again a unit vector, then necessarily
\begin{equation}\label{magic}
|\omega_1+\omega_2|^2+|\omega_2+\omega_3|^2+|\omega_3+\omega_1|^2=4.
\end{equation}
To see why this is true, one simply squares the assumption $|\omega_1+\omega_2+\omega_3|=1$ and expands the left-hand side of \eqref{magic}. 
An application of the Cauchy--Schwarz inequality together with identity \eqref{magic} reduces the analysis to antipodally symmetric functions, and at the same time neutralizes the singularity of the convolution measure at the origin. This reduces the $4$-linear problem on $f$ to a bilinear problem on its square $f^2$. More precisely, one is left with establishing a monotonicity property for  the quadratic form 
$$H(g)=\int_{(\mathbb{S}^2)^2} \overline{g(\omega)} g(\nu) |\omega-\nu| \d\sigma_\omega \d\sigma_\nu,$$
where $g=f^2$ is now assumed to be merely integrable.
This in turn is accomplished via spectral analysis. If $c$ denotes the mean value of $g$ over the sphere and ${\bf 1}$ denotes the constant function equal to 1, one wants to show that $H(g)\leq H(c {\bf 1}).$
The crucial observation is that the quadratic form $H$ is diagonal in a suitable basis. 
In fact, expanding~$g=\sum_{k\geq 0} Y_k$ in spherical harmonics, we have
\begin{equation}\label{estimateHg}
H(g)
=2\pi\sum_{k\geq 0} \Lambda_k \|Y_k\|_{L^2(\mathbb{S}^2)}^2
\leq 2\pi \Lambda_0 \|Y_0\|_{L^2(\mathbb{S}^2)}^2=H(c {\bf 1}),
\end{equation}
where the eigenvalues $\Lambda_k$ can be computed via the Funk--Hecke formula \cite[p.~247]{EMOT}. 
It turns out that $\Lambda_k<0$ when $k\geq 1$, we refer the reader to \cite{F} for the full details. This approach was extended in \cite{COS} to establish sharp $L^2(\sigma_{d-1})\to L^4$ extension estimates on~$\mathbb{S}^{d-1}$ for~\mbox{$d=4,5,6,7$}. 
Table \ref{tab:hresult}
indicates the signs of the corresponding coefficients $\Lambda_k=\Lambda_k (\mathbb{S}^{d-1})$.

\begin{table}[htbp] 
\centering \begin{tabular}{c cccccccc} 
$$$$$$$$\\
\hline 
&$\Lambda_0$&$\Lambda_1$&$\Lambda_2$&$\Lambda_3$&$\Lambda_4$&$\Lambda_5$&$\Lambda_6$&$\ldots$\\
\hline
$d=3$ & + & $-$ & $-$& $-$& $-$& $-$& $-$&$\ldots$\\	
$d=4$&+&0& $-$&0&$-$&0&$-$&$\ldots$\\ 
$d=5,6,7$	& + & + & $-$& $-$& $-$& $-$& $-$&$\ldots$\\
$d\geq 8$& + & + & {{+}} & $\ast$& $\ast$& $\ast$&$\ast$& $\ldots$\\[1ex] 
\hline \end{tabular}
\caption{Signs of the coefficients $\Lambda_k(\mathbb{S}^{d-1})$} 
 \label{tab:hresult} 
\end{table}

 Note that the sum in \eqref{estimateHg} ranges over even values of $k$ only since the function $f$, and therefore $g=f^2$, is  assumed to be antipodally symmetric.
In particular, the sign of the coefficient~$\Lambda_1$ is of no importance to the analysis.
However,  one starts to observe that $\Lambda_2 (\mathbb{S}^{d-1})>0$ if~$d\geq 8$, and this is the reason for the failure of this method of proof.

\noindent Threefold convolution measures can also be computed via delta calculus, at the expense of possibly more complicated expressions. For instance, the convolution $\sigma_2\ast\sigma_2\ast\sigma_2$ 
is a radial function supported on the ball of radius 3 centered at the origin, given by the expression
\begin{displaymath}
(\sigma_2\ast\sigma_2\ast\sigma_2)(\xi) = \left\{ \begin{array}{ll}
8\pi^2, & \textrm{if $0\leq |\xi|\leq 1$,}\\
4\pi^2\Big(-1+\frac{3}{|\xi|}\Big), & \textrm{if $1\leq |\xi|\leq 3$.}
\end{array} \right.
\end{displaymath}
Notice that this defines a bounded, continuous function which is constant inside the unit ball, and decreases to zero on the annulus $\{1\leq |\xi| \leq 3\}$.
The lowest dimensional endpoint case $(d,p)=(2,6)$ holds some hidden surprises. 
First of all, the estimate translates into an inequality involving a 6-linear form.
The convolution $\sigma_1\ast\sigma_1\ast\sigma_1$ defines a radial function given by a complicated integral expression,
see \cite[Lemma 8]{CFOST}, which is better illustrated in Figure \ref{fig:tripleconvS1plot} below.

\begin{figure}[htbp]
  \centering
  \includegraphics[height=5cm]{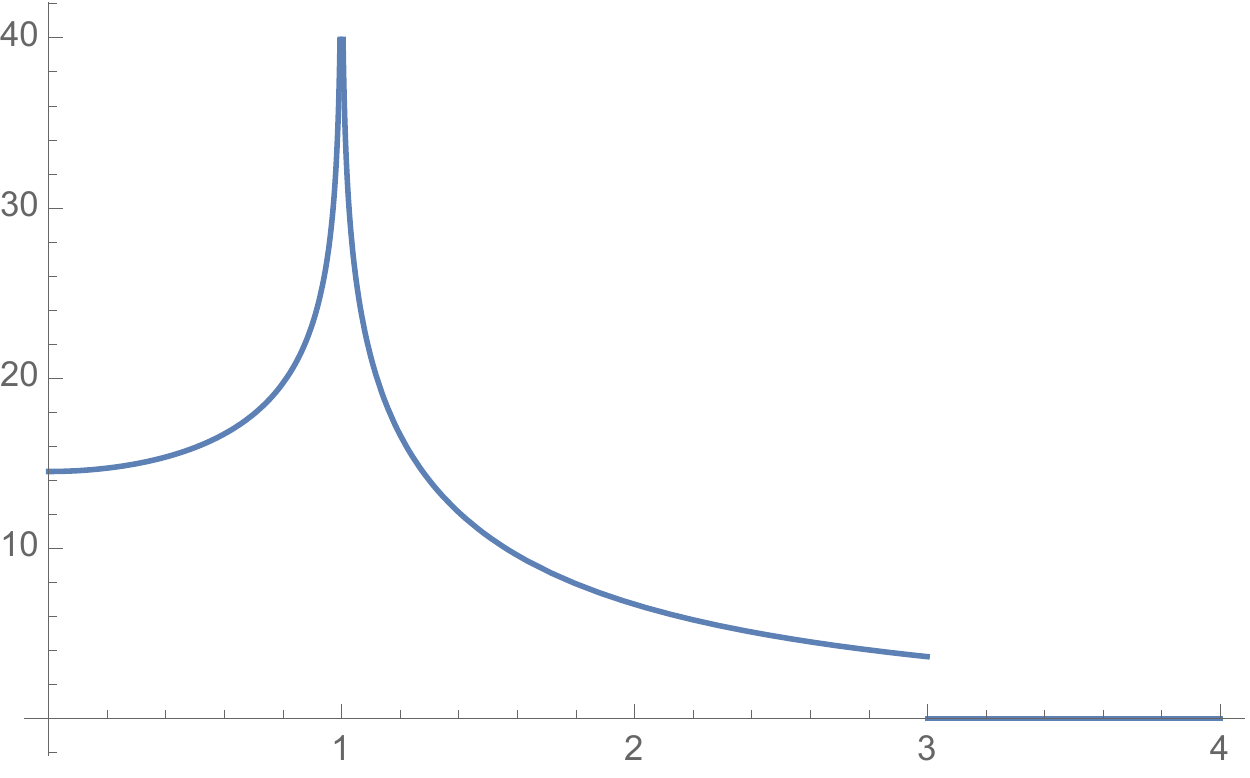} 
    \caption{Plot of the function $r\mapsto (\sigma_1 \ast \sigma_1 \ast \sigma_1) (r)$.}
\label{fig:tripleconvS1plot}
\end{figure}
\noindent A difficulty is that the convolution $\sigma_1\ast\sigma_1\ast\sigma_1$  now blows up along a whole  circle, and not just at one point. 
One would still like to reduce matters to a two-step analysis, and a possible program is as follows.
One first reduces the 6-linear analysis to a trilinear analysis on squares, and then performs a spectral analysis of the trilinear form
 \begin{equation}\label{defT}
T(h_1, h_2, h_3) 
= \int_{(\mathbb{S}^1)^6} 
h_1(\omega_1) h_2(\omega_2) h_3(\omega_3) \big(|\omega_4+\omega_5+\omega_6|^2-1\big)\,\d\Sigma_{\vec{\omega}},
\end{equation}
where the measure $\Sigma$ is given by
\begin{equation}\label{DefSigma}
\d\Sigma_{\vec{\omega}}=\ddirac{\omega_1+\omega_2+\omega_3+\omega_4+\omega_5+\omega_6}\,\d\sigma_{\omega_1}\,\d\sigma_{\omega_2}\,\d\sigma_{\omega_3}\,\d\sigma_{\omega_4}\,\d\sigma_{\omega_5}\,\d\sigma_{\omega_6}.
\end{equation}
This is later specialized to the case $h_1=h_2=h_3=f^2$. 
Let us formulate these steps in a more precise manner:

\begin{step}\label{task1}
Let $f\in L^2(\mathbb{S}^1)$ be nonnegative and antipodally symmetric. Then: 
\begin{multline*}\label{CSConj}
\int_{(\mathbb{S}^1)^6} f(\omega_1)f(\omega_2)f(\omega_3)f(\omega_4)f(\omega_5)f(\omega_6)
\big(|\omega_4+\omega_5+\omega_6|^2-1\big)\,\d\Sigma_{\vec{\omega}} \\
\leq \int_{(\mathbb{S}^1)^6} f(\omega_1)^2f(\omega_2)^2f(\omega_3)^2\big(|\omega_4+\omega_5+\omega_6|^2-1\big)\,\d\Sigma_{\vec{\omega}}.
\end{multline*}
\end{step}

\begin{step}\label{task2}
Let $h \in L^1(\mathbb{S}^1)$ be a nonnegative and antipodally symmetric function. 
Let $c$ denote the mean value of $h$ over $\mathbb{S}^1$. 
Then 
$T(h,h,h) \leq c^3 T({\bf 1},{\bf 1},{\bf 1})$,
with equality if and only if $h$ is constant.
\end{step}

\noindent Step \ref{task1} is an open problem, posed as a conjecture in \cite{CFOST}. 
Step \ref{task2} is the subject of papers~\cite{CFOST, OST}. 
In short, one decomposes $h=c+g$ where $g$ has mean zero, and analyzes each of the summands in
$$T(h,h,h)=T(c,c,c)+3T(c,c,g)+3T(c,g,g)+T(g,g,g)$$
separately. 
The linear term in $g$ vanishes for symmetry considerations, the bilinear term is nonpositive and the trilinear term can be controlled in absolute value by the bilinear term. The proof of these statements relies on expanding the integrand in \eqref{defT} in Fourier basis, and appealing to delicate estimates for integrals of sixfold products of Bessel functions. We refer the reader to the original papers for details, and proceed to discuss a related inequality which can be put in sharp form by invoking more elementary estimates for integrals of Bessel functions.

\subsection{Estimates in mixed norm spaces}\label{sec:MixedNorm}
The connection between Bessel functions and  spherical harmonics is easily seen via 
 the formula

\begin{equation}\label{second}
\widehat{Y_k\sigma}(x)= i^k J_{\frac{d}2-1+k}(|x|) |x|^{1-\frac{d}2} Y_k\Big(\frac{x}{|x|}\Big),\;\;\;(x\in\R^d)
\end{equation}
which follows from an application of the Funk--Hecke formula together with Rodrigues formula for Gegenbauer polynomials, see \cite{SW, W}.
Spherical harmonic expansions played an important role in the recent work of C\'ordoba \cite{Co} on certain Fourier extension inequalities in mixed norm spaces. 
The following inequality, which had already appeared in the Ph.D. thesis of Vega \cite{V}, was reproved in \cite{Co} using different methods:
For $d\geq 2$ and $q>\frac{2d}{d-1}$, there exists $C_{d,q}<\infty$ such that
\begin{equation}\label{Vega}
\|\widehat{f\sigma}\|_{L^q_{{\text{rad}}} L^2_{{\text{ang}}}(\R^d)}
\leq C_{d,q} \|f\|_{L^2(\mathbb{S}^{d-1})}.
\end{equation}
 Here the norm in $L^q_{\text{rad}} L^2_{\text{ang}}(\R^d)$ is given by the integral
$$\Big(\int_0^\infty\Big(\int_{\mathbb{S}^{d-1}} |\widehat{f\sigma}(r\omega)|^2 \d\sigma_\omega\Big)^{\frac q2} r^{d-1} \d r\Big)^{\frac 1q}.$$
Note that inequality \eqref{Vega} follows in a simple way from H\"older's inequality for those exponents $q$ for which the (adjoint) Tomas--Stein inequality is known to hold. 
We now indicate a possible path to obtain the sharp form of a number of instances of inequality \eqref{Vega}.
Given~$f\in L^2(\mathbb{S}^{d-1})$, we expand it in normalized spherical harmonics
$$f=\sum_{k\geq 0} a_k Y_k,$$
with each $\|Y_k\|_{L^2}=1$.
Appealing to formula \eqref{second} and to orthogonality of the $\{Y_k\}$, we \mbox{see that}
$$\int_{\mathbb{S}^{d-1}} |\widehat{f\sigma}(r\omega)|^2 \d\sigma_\omega
=\sum_{k\geq 0} |a_k|^2 |J_{\frac{d}{2}-1+k}(r)|^2 r^{2-d}.$$
The $q$-th power of the left-hand side of inequality \eqref{Vega} can thus be rewritten as
\begin{equation}\label{LHS}
  \int_0^\infty \Big(\sum_{k\geq 0} |a_k|^2 |J_{\frac{d}{2}-1+k}(r)|^2\Big)^{\frac q2} r^{(1-\frac d2)q} r^{d-1} \d r,
\end{equation}
which is the starting point for the analysis in \cite{Co}. Whenever the exponent $q$ is an even integer, we can take an alternative route and in principle obtain a sharp inequality. Let us illustrate this in the  case $(d,q)=(2,6)$, which corresponds to a weaker form of the~\mbox{$L^2(\sigma_1)\to L^6$} adjoint Tomas--Stein inequality discussed in the previous subsection. 
In this case, the integral \eqref{LHS} can be rewritten as
$$\sum_{k,\ell,m\geq 0} |a_k|^2 |a_\ell|^2 |a_m|^2 I(k,\ell,m),$$
where the integrals $I(k,\ell,m)$ are defined as
$$I(k,\ell,m)=\int_0^\infty  J^2_{k}(r)J^2_{\ell}(r) J^2_{m}(r) r \d r.$$
These integrals obey the following monotonicity property:
\begin{proposition}\label{task3}
Let $k,\ell,m$ be nonnegative integers. Then:
$$I(k,\ell,m)\leq I(0,0,0),$$
and equality holds if and only if $k=\ell=m=0$.
\end{proposition}
\noindent
 As a consequence, we  obtain a sharp form of this particular instance of inequality \eqref{Vega}:
\begin{multline*}
\|\widehat{f\sigma}\|^6_{L^6_{{\text{rad}}} L^2_{{\text{ang}}}(\R^2)}
=\sum_{k,\ell,m\geq 0} |a_k|^2 |a_\ell|^2 |a_m|^2 I(k,\ell,m) \\
\leq I(0,0,0)\sum_{k,\ell,m\geq 0} |a_k|^2 |a_\ell|^2 |a_m|^2 
=I(0,0,0) \|f\|_{L^2(\mathbb{S}^1)}^6,
\end{multline*}
with constants being the unique extremizers.

\begin{proof}[Proof of Proposition \ref{task3}]
Reasoning as in \cite[\S 3]{CFOST}, we see that
$$I(k,\ell,m)
=\int_{\R^2} \widehat{e_{k}\sigma} \overline{\widehat{e_{k}\sigma}} 
\widehat{e_{\ell}\sigma} \overline{\widehat{e_{\ell}\sigma}} 
\widehat{e_{m}\sigma} \overline{\widehat{e_{m}\sigma}} \d x
=\int_{(\mathbb{S}^1)^6} 
(\omega_1 \overline{\omega_2})^k (\omega_3 \overline{\omega_4})^\ell (\omega_5 \overline{\omega_6})^m 
\d\Sigma_{\vec{\omega}},$$
where the function $e_n:\mathbb{S}^1\to\Co$ is defined via $e_n(\omega)=\omega^n$ and the measure $\Sigma$ is given by~\eqref{DefSigma}.
The triangle inequality immediately implies
$$I(k,\ell,m)\leq I(0,0,0),$$
with equality if and only 
$$(\omega_1 \overline{\omega_2})^k (\omega_3 \overline{\omega_4})^\ell (\omega_5 \overline{\omega_6})^m=1,$$
for every $\omega_1,\ldots,\omega_6\in\mathbb{S}^1$ such that $\sum_{j=1}^6 \omega_j=0$.
One easily checks that this implies~\mbox{$k=\ell=m=0$}, and the proof is complete.
\end{proof}

\section{Towards a unifying picture}\label{sec:unify}

Table \ref{tab:extremizers} summarizes some of our discussion from the previous sections concerning the sharp form of a number of Fourier extension inequalities which can be recast as bilinear estimates for the appropriate convolution measures.

\begin{table}[htbp] 
\centering \begin{tabular}{c c c} 
$$$$$$$$\\
\hline 
Manifold&Fourier extension inequality&Extremizers\\
\hline
Paraboloid &$L^2(\mathbb{P}^2, \mu_2)\to L^4(\R^3)$&Gaussians\\
Cone&$L^2(\Gamma^3, \nu_3)\to L^4(\R^4)$&Exponentials\\
Hyperboloid&$L^2(\mathbb{H}^2, \lambda_2)\to L^4(\R^3)$&Do not exist\\
Sphere&$L^2(\mathbb{S}^2, \sigma_2)\to L^4(\R^3)$&Constants\\
\hline \end{tabular}
\caption{} 
 \label{tab:extremizers} 
\end{table}

\noindent The two-dimensional surfaces in question (paraboloid, hyperboloid and sphere) can all be obtained as intersections of the three-dimensional cone with appropriately chosen hyperplanes. Figure \ref{fig:conic} below illustrates this point. There, the ambient space $\R^4$ is endowed with coordinates~$(\xi,\tau)\in\R^{3+1}$, where $\xi=(\xi_1,\xi')\in\R^{1+2}$.
One cannot help noticing that the restriction of the exponential function $\exp({-|\xi|})$, which is an extremizer for the cone, to the different conic sections coincides with the corresponding extremizers given by Table \ref{tab:extremizers}. 
It is 
a constant function when $\tau=1$ (sphere, red in Fig. \ref{fig:conic})
and
a Gaussian when $\tau=-\xi_1+2$ (paraboloid, blue in Fig. \ref{fig:conic}).
Furthermore, it yields the function $\exp({-\sqrt{1+|\xi'|^2}})$ when $\tau=\sqrt{1+|\xi'|^2}$ (hyperboloid, yellow in Fig. \ref{fig:conic}).
Extemizers for the problem on the hyperboloid do not exist, but it was shown in \cite[Lemma 5.4]{Q2} that the function $f_a(\xi')=\exp({-a\sqrt{1+|\xi'|^2}})$ produces an extremizing sequence  $\{f_a/\|f_a\|_{L^2}\}$ as $a\to\infty$ for the $L^2(\lambda_2)\to L^4$ extension inequality on the hyperboloid.

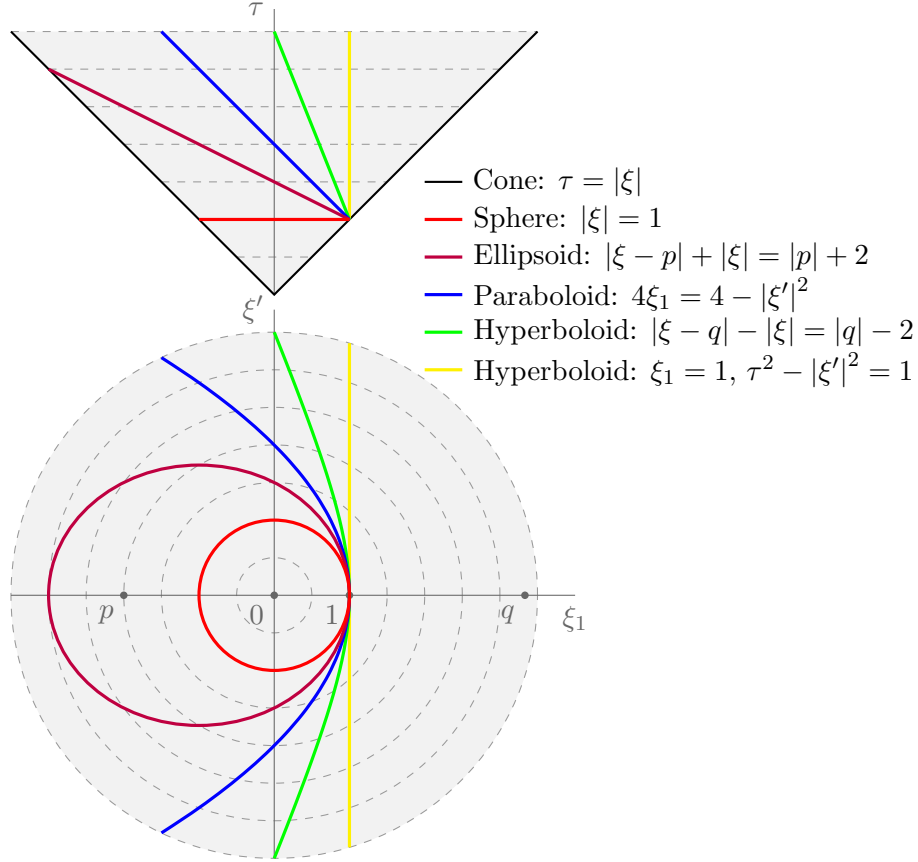
\begin{figure}[htbp]
  \centering

  \begin{tikzpicture}[scale = 1]

    \fill[black!5]
       (-3.5, 3.5) -- (0, 0) -- (3.5, 3.5) -- cycle
       (0, -4) circle (3.5);

    \foreach \l in {0.5, 1, ..., 3.6} {
      \draw[thin, dashed, black!40]
         (-\l, \l) -- (\l, \l)
         (0, -4) circle [radius = \l];
    }

    \draw[thin, black!60]
       (0, 0) -- (0, 3.8) node[left] {$\tau$}
       (0, -7.5) -- (0, -0.2) node[left] {$\xi'$}
       (-3.5, -4) -- (4, -4) node[below] {$\xi_1$};

    \fill[black!60]
       (3.333, -4) circle [radius = 0.05] node[below left] {$q$}
       (1, -4) circle [radius = 0.05] node[below left] {$1$}
       (0, -4) circle [radius = 0.05] node[below left] {$0$}
       (-2, -4) circle [radius = 0.05] node[below left] {$p$};

    \draw[thick, text=black]
       (2, 1.5) -- (2.5, 1.5) node[right] {Cone: $\tau = \abs{\xi}$}
       (-3.5, 3.5) -- (0, 0) -- (3.5, 3.5);

    \draw[very thick, yellow, text=black]
       (2, -1) -- (2.5, -1) node[right] {Hyperboloid: $\xi_1 = 1$, $\tau^2 - \abs{\xi'}^2 = 1$}
       (1, 1) -- (1, 3.5)
       (1, -7.354) -- (1, -0.646);

    \draw[very thick, green, text=black]
       (2, -0.5) -- (2.5, -0.5) node[right] {Hyperboloid: $\abs{\xi - q} - \abs{\xi} = \abs{q} - 2$}
       (0, 3.5) -- (1, 1)
       [domain=-3.5:3.5, smooth, variable=\t] plot ({(8.75-sqrt(5.25*\t*\t + 12.25))/5.25}, \t-4); 

    \draw[very thick, blue, text=black]
       (2, 0) -- (2.5, 0) node[right] {Paraboloid: $4 \xi_1 = 4 - \abs{\xi'}^2$}
       (-1.5, 3.5) -- (1, 1)
       [domain=-3.162:3.162, smooth, variable=\t] plot (1-\t*\t/4, \t-4);

    \draw[very thick, purple, text=black]
       (2, 0.5) -- (2.5, 0.5) node[right] {Ellipsoid: $\abs{\xi - p} + \abs{\xi} = \abs{p} + 2$}
       (-3, 3) -- (1, 1)
       (-1, -4) circle [x radius = 2, y radius = 1.732];

    \draw[very thick, red, text=black]
       (2, 1) -- (2.5, 1) node[right] {Sphere: $\abs{\xi} = 1$}
       (-1, 1) -- (1, 1)
       (0, -4) circle [radius = 1];

  \end{tikzpicture} 

  \caption{Conic sections}
  \label{fig:conic}
\end{figure}

\appendix
\section{Integration on manifolds using delta calculus}

Let $\mM$ be a smooth $k$-dimensional submanifold of $\R^d$, with~\mbox{$0 < k < d$}.
On $\mM$ one can  define a canonical measure $\sigma = \sigma_\mM$ which is naturally induced by the Euclidean metric structure of $\R^d$.
Integration on the manifold $\mM$ can be rigorously defined by means of differential forms \cite{Spi},
but the actual computation of integrals of the form
\begin{equation*}
  \int_\mM \varphi(x) \d\sigma
\end{equation*}
often appears to be a challenging task when $\mM$ is a nontrivial manifold.
The theory of distributions comes into help since these integrals can be viewed as pull-backs
of Dirac delta distributions \cite{Ho}.

Suppose that $\mM$ is (locally) implicitly described as the zero level set
of a $C^1$ function~$f$ defined on an open subset $\Omega$ of $\R^d$ and taking values in $\R^{d - k}$,
\begin{equation*}
  \mM = \mM_f := \insieme{x \in \Omega}{f(x) = 0},
\end{equation*}
and assume that the Jacobian matrix $Df(x)$ of the map $f$ has maximal rank at every point $x \in \Omega$.
Consider the usual Dirac delta distribution $\ddirac{\cdot}$ on $\R^{d - k}$ defined by
\begin{equation*} 
  \ip{\ddirac{\cdot}, \varphi} = \int_{\R^{d - k}} \ddirac{y} \varphi(y) \d y  = \varphi(0),
\end{equation*}
for any smooth test functions $\varphi$.
We would like to make sense of the composition $\ddirac{f(x)}$ as a distribution on $\Omega$.
There are several ways to define such a composition.
One possibility is to approximate the delta distribution with a family of smooth bump functions,
\begin{equation*}
  \gamma_\epsi(y) = \epsi^{k - d} \gamma(\epsi^{-1} y), \quad
  \text{where} \quad
  0 \le \gamma \in C^\infty_c(\R^{d - k}), \quad
  \int_{\R^{d - k}} \gamma(y) \d y = 1.
\end{equation*}
One is led to define 
\begin{equation*}
  \int_{\Omega} \ddirac{f(x)} \varphi(x) \d x = \ip{\ddirac{f}, \varphi}
  := \lim_{\epsi \to 0} \int_{\Omega} \gamma_\epsi\tonde{f(x)} \varphi(x) \d x, 
\end{equation*}
for every $\varphi \in C^\infty_c(\Omega)$.
This limit converges to the integral
\begin{equation} \label{eq:dddef}
  \int_{\Omega} \ddirac{f(x)} \varphi(x) \d x = \int_\mM \frac{\varphi(x)}{J_f(x)} \d\sigma,
\end{equation}
where $J_f(x)$ is the Jacobian determinant of the function $f$ at the point $x$,
\begin{equation*}
  J_f(x) = \sqrt{\det\tonde{Df(x)\cdot Df(x)^t}} = \abs{\d f_1(x) \wedge \d f_2(x) \wedge \dots \wedge \d f_{d - k}(x)},
\end{equation*}
and here $\wedge$ stands for the wedge product of differential forms.
Therefore integrals over the manifold $\mM$ can be expressed as integrals over $\R^d$ by means of delta distributions:
\begin{equation*}
  \int_\mM \varphi(x) \d\sigma = \int_{\R^d} \ddirac{f(x)} \varphi(x) J_f(x) \d x.
\end{equation*}

\noindent Let us now have a look at some simple but useful algebraic rules which follow easily from these definitions,
and allow for manipulation of  integrals with delta distributions.
These rules have been used in the previous sections to carry out explicit computations
of convolution  measures defined on the various manifolds considered there.

\subsection{Hypersurfaces}

In the codimension $1$ case, $k = d - 1$, we have that $\mM$ is a hypersurface defined by a scalar function.
Identity \eqref{eq:dddef} simplifies to
\begin{equation} \label{eq:hypersurf}
  \int_{\Omega} \ddirac{f(x)} \varphi(x) \d x = \int_\mM \frac{\varphi(x)}{\abs{\nabla f(x)}} \d\sigma.
\end{equation}

For example, in the case of the unit sphere $\mathbb{S}^{d-1}$ equipped with surface measure $\sigma$,
we have that  $f(x) = \abs{x} - 1$, $\abs{\nabla f(x)} = 1$, and
\begin{equation*}
  \int_{\mathbb{S}^{d-1}} \varphi(x) \d\sigma = \int_{\R^d} \ddirac{\abs{x} - 1} \varphi(x) \d x.
\end{equation*}

For another example, consider the region $\Omega_- = \insieme{x \in \Omega}{f(x) < 0}$, with boundary inside~$\Omega$ given by~$\mM$,
and with outgoing unit normal vector given by $\nu = \frac{\nabla f}{\abs{\nabla f}}$.
Combining the divergence theorem with formula~\eqref{eq:hypersurf}, we have the following: 
For any smooth vector field $V(x)$ with compact support in $\Omega$,
\begin{equation*}
  \int_{\Omega_-} (\nabla \cdot V)(x) \d x = \int_\mM V(x) \cdot \nu(x) \d\sigma = 
  \int_\Omega \ddirac{f(x)} \, V(x) \cdot \nabla f(x) \d x.
\end{equation*}

\subsection{Products of delta distributions}

We can make sense of the product of two delta distributions by simply setting
\begin{equation*}
  \ddirac{f(x)} \ddirac{g(x)} := \ddirac{f(x) \\ g(x)},
\end{equation*}
whenever the right-hand side is well-defined
as a distribution supported on the manifold $\mM_{(f,g)} = \mM_f \intersection \mM_g$.
We can also make sense of the integration of the distribution $\ddirac{f(x)}$ over the manifold $\mM_g$ as follows:
\begin{align}
  \int_{\mM_g} \ddirac{f(x)} \varphi(x) \d\sigma_{\mM_g} 
  &=\int_\Omega \ddirac{f(x) \\ g(x)} \varphi(x) J_g(x) \d x \label{prodrule}\\
 & = \int_{\mM_f \intersection \mM_g} \varphi(x) \frac{J_g(x)}{J_{(f,g)}(x)} \d\sigma_{\mM_{(f,g)}}.\notag
\end{align}
In the case of codimension $d - k > 1$, 
the delta distribution can always be viewed (locally) as a product of $d - k$ delta distributions on hypersurfaces.

For example, let $\sigma$ denote the surface measure on the $(d-1)$-dimensional unit sphere in~$\R^d$,
and let $h$ be a smooth function on the unit sphere. 
The convolution $h \sigma \ast h \sigma$ is supported on the ball of radius $2$ centered at the origin,
and its value at a point $x$ inside that ball can be written as an integral
over the $(d-2)$-dimensional sphere $\Gamma_x$ obtained as the intersection of the unit sphere with its translate by $x$.
\begin{figure}[htbp]
  \centering
  \begin{tikzpicture} [scale = 2]
    \draw[thick, ->] (0, 0) node[below] {$0$} -- (1.6, 0.8) node[below] {$x$};
    \draw[thick, ->] (0, 0) -- node[midway, above] {$1$} (0.6, 0.8) node[above right] {$y$};
    \draw[dashed] (0.6, 0.8) -- node[midway, above] {$1$} (1.6, 0.8);
    \draw[thick, blue] (0, 0) circle [radius = 1];
    \draw[thick, blue] (1.6, 0.8) circle [radius = 1];
    \draw[very thick, red] (0.6, 0.8) -- (1, 0);
    \fill (0, 0) circle [radius=0.03] (0.6, 0.8) circle [radius=0.03] (1.6, 0.8) circle [radius=0.03];
    \draw[black!50, ->] (1.5, -0.5) node[text=black, right] {$\Gamma_x$} .. controls (0.5, -0.5) .. (0.85, 0.3);  

  \end{tikzpicture}
  \caption{$\Gamma_x = \mathbb{S}^{d-1} \intersection (x + \mathbb{S}^{d-1})$}
  \label{fig:Gammax}
\end{figure}
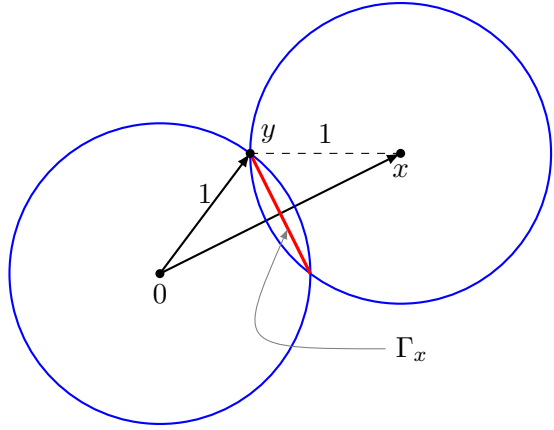
The sphere $\Gamma_x$ has radius~$\sqrt{1 - \frac{\abs{x}^2}4}$.
We can write~$\Gamma_x = \mM_f \intersection \mM_g$, with $f(y) = \abs{y} - 1$ and $g(y) = \abs{x - y} - 1$.
If $y \in \Gamma_x$, then 
\begin{equation*}
 J_{(f,g)}(y) = \abs{\d f(y) \wedge \d g(y)} = \abs{\frac{x - y}{\abs{x - y}} \wedge \frac{y}{\abs{y}}} = \abs{x \wedge y} = 
  \abs{x} \sqrt{1 - \frac{\abs{x}^2}4}.
\end{equation*}
Using formula~\eqref{prodrule}, we obtain a generalization of formula~\eqref{2foldsphconv}: If $|x|\leq 2$, then 
\begin{align*}
  (h \sigma \ast h \sigma) (x) 
  =& \int_{\R^d} \ddirac{\abs{x - y} -1 \\ \abs{y} - 1} h(x - y) h(y) \d y  \\
  =& \int_{\Gamma_x} \frac{h(x-y) h(y)}{J_{(f,g)}(y)} \d\sigma_y 
  =\frac2{\abs{x} \sqrt{4 - \abs{x}^2}} \int_{\Gamma_x} h(x-y) h(y) \d\sigma_y \\
  =& \frac{V_{d-2}}{\abs{x}} \tonde{1 - \frac{\abs{x}^2}4}^{\frac{d - 3}2} \fint_{\Gamma_x} h(x-y) h(y) \d\sigma_y, 
\end{align*}
where $V_{d-2}$ is the surface volume of the $(d-2)$-dimensional unit sphere
and $\fint$ denotes the averaged integral.

\subsection{Multiplication by scalar functions}

Let $\alpha:\Omega\to\R$ be a positive $C^1$ scalar function.
Then, on the manifold $\mM_{\alpha f} = \mM_f$, we have that $J_{\alpha f} = \alpha^{d-k} J_f$,
and consequently
\begin{equation}\label{scalarmult}
  \ddirac{\alpha(x) f(x)} = \alpha(x)^{-d + k} \ddirac{f(x)}.
\end{equation}
In particular, if $\alpha$ and $\beta$ are smooth positive scalar functions which coincide on $\mM_f$,
then~\mbox{$\ddirac{\alpha f} = \ddirac{\beta f}$}.
For example, on the unit sphere, we have that
\begin{equation*}
  \ddirac{\abs{x}^2 - 1} = \frac1{\abs{x} + 1} \ddirac{\abs{x} - 1} = \frac12 \ddirac{\abs{x} - 1}.
\end{equation*}
In a similar way, on the null cone, we have that
\begin{equation*}
  \ddirac{t^2 - \abs{x}^2} = \frac{\ddirac{t - \abs{x}}}{t + \abs{x}}  - \frac{\ddirac{t + \abs{x}}}{t - \abs{x}} =
  \frac{\ddirac{t - \abs{x}}}{2 \abs{x}} + \frac{\ddirac{t + \abs{x}}}{2\abs{x}}. 
\end{equation*}

\subsection{Change of variables}

Suppose that $x = \Psi(y)$ is a local diffeomorphism on $\R^d$.
As for any distribution, we have the usual change of variables rule,
\begin{equation}\label{changevars1}
  \int_\Omega \ddirac{f(x)} \varphi(x) \d x = 
  \int_{\Psi^{-1}(\Omega)} \ddirac{f(\Psi(y))} \varphi(\Psi(y)) \abs{\det D\Psi(y)} \d y. 
\end{equation}

\noindent If instead we consider the composition $g(x) = \Phi(f(x))$,
with $\Phi$ a local diffeomorphism defined on a neighborhood of $0$ in $\R^{d - k}$ and satisfying $\Phi(0) = 0$,
then $\mM_g = \mM_f$, and
\begin{equation*}
  \ddirac{\Phi(f(x))} = \abs{\det D\Phi(f(x))}^{-1} \ddirac{f(x)}.
\end{equation*}
In particular, if $L$ is a nonsingular $d \times d$ matrix and $M$ is a nonsingular $(d - k) \times (d - k)$ matrix, 
then we have that
\begin{equation*}
  \int_{L^{-1}(\Omega)} \ddirac{M f(L x)} \varphi(x) \d x 
  = \frac 1{\abs{{\det L}\,{\det M}}} \int_\Omega \ddirac{f(y)} \varphi(L^{-1} y) \d y.
\end{equation*}

\section*{Acknowledgements}
The second author expresses his gratitude to Dmitriy Bilyk, Feng Dai, Vladimir Temlyakov and  Sergey Tikhonov for organizing 
the workshop
{\it Function spaces and high-dimensional approximation}
at the 
Centre de Recerca Matem\`atica-ICREA, May 2016,
whose hospitality is greatly appreciated.

\end{document}